\documentclass[reqno,12pt]{amsart}
\usepackage{graphicx,color}
\usepackage{latexsym}
\usepackage{graphicx}
\usepackage[english]{babel}
\usepackage[utf8]{inputenc}
\usepackage{wrapfig}
\usepackage{amsthm,amsmath, amssymb,amsfonts,esint}
\usepackage{mathtools}
\usepackage{setspace}
\usepackage{accents}
\usepackage{hyperref}%para ir de uma referência a sua localização no texto

\usepackage{layout,textcomp}
\theoremstyle{plain}
\newtheorem{proposition}{Proposition}
\newtheorem{theorem}{Theorem }

\newtheorem{corollary}{Corollary}

\newcommand{\Ric}{\mbox{Ric}}

\theoremstyle{remark}
\newtheorem{remark}{Remark}
\newtheorem{example}{Example}

\DeclareMathOperator{\Ricst}{\accentset{\circ}{\Ric}}

\DeclareMathOperator{\sen}{sen}

\DeclareMathOperator{\Hess}{Hess}
\DeclareMathOperator{\Div}{div}

\DeclareMathOperator{\inte}{int}

\begin{document}

\title[On static manifolds with boundary]{On static manifolds satisfying an overdetermined Robin type condition on the boundary} 

\author{Tiarlos Cruz}
\address{
Universidade Federal de Alagoas\\
Instituto de Matem\'atica\\
Macei\'o, AL -  57072-970, Brazil}
\email{cicero.cruz@im.ufal.br}
%\thanks{T. Cruz has been partially suported by  CNPq/Brazil grant 311803/2019-9.}

\author{Ivaldo Nunes}
\address{
  Universidade Federal do Maranhão\\
  Departamento de Matemática\\
 São Luís, MA - 65080 - 805, Brazil}
\email{ivaldo.nunes@ufma.br}

\begin{abstract}
In this work, we consider static manifolds $M$ with nonempty boundary $\partial M$. In this case, we suppose that the potential $V$ also satisfies an overdetermined Robin type condition on $\partial M$. We prove a rigidity theorem for the Euclidean closed unit ball $B^3$ in $\mathbb{R}^3$. More precisely, we give a sharp upper bound for the area of the zero set $\Sigma=V^{-1}(0)$ of the potential $V$, when $\Sigma$ is connected and intersects $\partial M$. We also consider the case where $\Sigma=V^{-1}(0)$ does not intersect $\partial M$. 
\end{abstract}

\maketitle

\section{Introduction}

The following equation
\begin{equation}\label{staticequations1}
\Hess_g V - \Delta_gVg - V\Ric_g=0,   
\end{equation}
defined on a Riemanniann manifold $(M^n,g)$, is a very  interesting equation that appears naturally when we consider the problem of deforming the scalar curvature of $M$ (see \cite{FM, Cor}) and also in the context of general relativity (see \cite{FM, Cor, A, CLM}). In the former context, equation \eqref{staticequations1}  
defines the formal adjoint of the linearization of the scalar curvature.

A Riemannian manifold $(M^n,g)$ which admits a nontrivial smooth solution $V: M\to\mathbb{R}$ to \eqref{staticequations1} is called a \textit{static} manifold. That name comes from general relativity since such a manifolds are related to the concept of static spacetimes. The function $V$ is called a \textit{static} potential. We refer the reader to \cite{Wald, Cor} for details.

A very interesting question is that of classifying static manifolds, specially in dimension three. There are several classification results in the literature. See \cite{K,L,A,AM,BM,BCM,BM2} and references therein. 

The following result, proved by Shen \cite{Shen} (see also \cite{BGH}), gives a sharp upper bound for the area of the zero set of the static potential of a static manifold of positive scalar curvature in the case where the zero set is connected. 

\begin{theorem}[Shen \cite{Shen}]\label{Shenthm}
Let $(M^3,g)$ be an oriented compact static manifold with static potential $V: M\to \mathbb{R}$ . Suppose the scalar curvature of $M$ satisfies $R_g= 6$. If $\Sigma=V^{-1}(0)$ is connected, then $\Sigma$ is a two-sphere and 
\begin{equation}\label{Shen_areaestimate}
|\Sigma|\leq 4\pi.
\end{equation}
Moreover, equality holds if and only if $(M^3,g)$ is isometric to the standard three-sphere $(S^3,g_{can})$ of radius 1.
\end{theorem}

Recently, motivated by \cite{FM,KW}, the first author and Vitório \cite{CV} considered the problem of prescribing the scalar curvature and the mean curvature of a Riemannian manifold with boundary. In the spirit of \cite{FM}, they considered the scalar curvature together with the mean curvature as a functional $g\mapsto (R_g,H_g)$ defined on the space of Riemannian metrics on a manifold $M$ with nonempty boundary. The formal adjoint of the linearization of this operator defines the following equations:

\begin{equation}\label{staticequations2}
\left\{
   \begin{array}{rcl}
\Hess_{g}V-(\Delta_{g} V){g}-V\Ric_{g} &= &0\quad\mbox{in}\quad M\\
\dfrac{\partial V}{\partial \nu}g-V\Pi_{g}& = &0\quad\mbox{on}\quad\partial M,
\end{array}
   \right.
\end{equation}
where  $\nu$ is the outward unit normal vector field to $\partial M$ and $\Pi_g$ is the second fundamental form of $\partial M$ with respect to $\nu$. Our convention here for $\Pi_g$ is such that the unit sphere have positive mean curvature with respect to the outward unit normal vector. 

We note that equations \eqref{staticequations2} are the analogues of \eqref{staticequations1} for manifolds with nonempty boundary and they also appear in the context of general relativity (see, for example, \cite{AdL}). Here, we will adopt the terminology introduced in \cite{AdL} and we will say that a Riemannian manifold $(M^n,g)$ with nonempty boundary is a \textit{static manifold with boundary} if there exists a non-trivial smooth function $V:M\to \mathbb{R}$, also called a \textit{static potential}, solution to \eqref{staticequations2}. 

A basic example of a compact static manifold with boundary is the Euclidean closed unit ball $B^n\subset\mathbb{R}^n$ where any static potential is given by $V(x)=\langle x,v\rangle$, $x\in B^n$, for some vector $v\in\mathbb{R}^n\setminus\{0\}$ (see, for example, \cite{HH}, Proposition 4.2). We note that in this case we have that the zero set of $V$ is equal to the flat closed unit disk $D^{n-1}=\{x\in\mathbb{R}^n:\langle x,v\rangle=0, |x|\leq 1\}$. We refer the reader to \cite{AdL} for basic examples of noncompact static manifolds with boundary.

In this work, we are interested in compact static manifolds with boundary with zero scalar curvature and positive mean curvature on the boundary. Our first result is the following whose item (i) is the analogue of Shen's theorem for this setting.

\begin{theorem}\label{shentypetheorem}
Let $(M^3,g)$ be an oriented compact static manifold with boundary such that $R_g=0$ and $H_g=2$. Let $V: M\to \mathbb{R}$ be the static potential of $M$. Suppose that $\Sigma=V^{-1}(0)$ is connected . Then
\begin{itemize}
\item[(i)] If $\Sigma\cap \partial M \neq \emptyset$, then $\Sigma$ is a free boundary totally geodesic two-disk and 
\begin{equation}\label{areaestimate2}
|\Sigma|\leq \pi.    
\end{equation}
Moreover, in this case, equality holds if and only if $(M^3,g)$ is isometric to the Euclidean unit ball $(B^3,\delta)$ and $V$ is given by $V(x)=\langle x,v\rangle$ for some vector $v\in\mathbb{R}^3\setminus\{0\}$.
\item[(ii)] If $\Sigma\cap \partial M=\emptyset$, then $\Sigma$ is a totally geodesic two-sphere and 
\begin{equation}\label{areaestimate1}
|\Sigma|<2\pi.
\end{equation}

\end{itemize}
\end{theorem}

We say a surface  $\Sigma\subset M$ with nonempty boundary $\partial \Sigma\neq \emptyset$ is \textit{free boundary} if $\partial\Sigma\subset \partial M$ and $\Sigma$ meets $\partial M$ orthogonally along $\partial\Sigma$.\\

The following example shows that case $(ii)$ in Theorem \ref{shentypetheorem} can occur.

\begin{example}\label{exampleSchw}
Given $m>0$, consider the three-dimensional Schwarzschild space $([2m,+\infty)\times \mathbb S^2,g_m)$ where
$$
 g_m=\left(1-\dfrac{2m}{r}\right)^{-1}dr^2+r^2\overline{g}.
$$
Here, $\overline{g}$ is the standard metric of the two-sphere $\mathbb S^2$ of radius $1$.

The function $V(r)=(1-2m/r)^{1/2}$ defines a static potential on $([2m,+\infty),g_m)$. Let $\Sigma_r=\{r\}\times S^2$ be a slice of this space. A simple computation shows that $\Sigma_r$ is totally umbilical with mean curvature given by $2V(r)/r$. 

Moreover, we have that $\partial V/\partial \nu_r=m/r^2$ where $\nu_r$ is the unit normal to $\Sigma_r$ parallel to $\partial_r$.  Thus, the equation
\begin{equation}\label{eqexample1}
\dfrac{\partial V}{\partial \nu_r}g_m - V\Pi_{g_m} = 0 \ \ \mbox{on $\Sigma_r$} 
\end{equation}
becomes
\begin{equation}\label{eqexample2}
m=V^2r.
\end{equation}
This implies that $r=3m$. Therefore $V$ solves \eqref{eqexample1} on $\Sigma_{3m}$. We note that the slice $\Sigma_{3m}$ is related to the notion of photon sphere in general relativity. In fact, $\Sigma_{3m}$ is a space-like slice of the photon sphere $\mathbb{R}\times \Sigma_{3m}$ in the static spherically symmetric Schwarzschild black hole spacetime of mass $m>0$.

Now, if we normalize to make the mean curvature of $\Sigma_{3m}$ equal to $2$ we get that 
$$
m=\dfrac{1}{3\sqrt{3}}.
$$

Note that in this case, we have that $\Sigma=V^{-1}(0)$ has area equal to $16\pi m^2=16\pi/27<2\pi$. In order to conclude the example we have just to reflect $[2m,3m]\times S^2$, $g_m$ and $V$ along $\Sigma$. Here, the potential is reflected with the opposite sign around $\Sigma$.

\end{example}

It is natural to ask about the uniqueness of the compact static manifold with boundary given in Example \ref{exampleSchw}. In our next result, we obtain some progress on this question. 

\begin{theorem}\label{photonsphererigidity}
Let $(M^3,g)$ be an orientable compact static manifold with boundary with static potential $V:M\to \mathbb{R}$. Suppose that $R_g=0$, $H_g>0$, $\Sigma=V^{-1}(0)$ is connected and $\Sigma\subset \inte M$. Let $\Omega$ be a connected component of $M\setminus \Sigma$. Then
\begin{itemize}
    \item [(i)] There is only one component of $\partial\Omega$ with positive mean curvature.
    \item[(ii)] If $S$ denotes the component of $\partial \Omega$ with positive mean curvature, then $S$ is a two-sphere and 
    \begin{equation}\label{psareaestimate}
    H_S^2|S|\leq \dfrac{16\pi}{3}.    
    \end{equation}
    Moreover, equality holds if and only if $(\Omega,g)$ is isometric to $([2m,3m]\times S^2,g_m)$, where $m=2/(3\sqrt{3}H_S)$, and $V$ is given by $V(r)=(1-2m/r)^{1/2}$.
\end{itemize}
\end{theorem}

\begin{remark}
Let $(M^n,g)$ be a static manifold with boundary with static potential $V:M\to \mathbb{R}$. It is known (see, for example, \cite{Cor}) that we can construct a static spacetime $(\widehat{M}^{n+1},\widehat{g})$ from $(M,g)$ and $V$ by defining $\widehat{M}=\mathbb{R}\times \left(M\setminus V^{-1}(0)\right)$ and $\widehat{g}=-V^2dt^2+g$. Let $S$ be a connected component $\partial M\setminus V^{-1}(0)$. Since $V$ satisfies \eqref{staticequations4} on $\partial M$, we have that $\widehat{S}=\mathbb{R}\times S$ is a time-like totally umbilical hypersurface in the spacetime $(\widehat{M},\widehat{g})$. Thus, $\widehat{S}$ is a photon surface\footnote{A photon surface in a spacetime $(\widehat{M}^{n+1},\widehat{g})$ is an embedded hypersurface $\widehat{S}\subset\widehat{M}$ such that any null geodesic initially tangent to $\widehat{S}$ remains tangent to $\widehat{S}$ as long it exists. It is known that $\widehat{S}$ is a photon surface if and only if $\widehat{S}$ is totally umbilical. For details, we refer the reader to \cite{CVE}, Theorem II.1, and \cite{Pe}, Proposition 1.} in $(\widehat{M},\widehat{g})$. It is worth to note that, recently, Cederbaum \cite{C1} established the uniqueness of the Schwarzchild spacetime among all static, asymptotically flat solutions to the vacuum Einstein equations which admits a single photon sphere\footnote{A photon sphere in a static spacetime $(\widehat{M}^{n+1}=\mathbb{R}\times M^n, \widehat{g}=-V^2dt^2+g)$ is a photon surface $\widehat{S}$ such that the lapse function $V$ is constant on each connected component of $\widehat{S}$. We refer the reader to \cite{CVE, C1, CG1, CG3} for details about photon surfaces and photon spheres.}. This result was later generalized by Cederbaum and Galloway \cite{CG1} to allow multiple photon spheres. For more results involving photon spheres we refer the reader to \cite{CG1, CG2, GW, J, Shoom, TSI1, TSI2, Y, YL1, YL2, Yo}. 

The problem of classifying static manifolds with boundary is a very natural one since \eqref{staticequations2} are overdetermined equations that impose some restrictions on the geometry of the manifold and its boundary. The above relationship between static manifolds with boundary and the notions of static spacetimes and photon surfaces makes the problem even more interesting, mainly in dimension $n=3$.
\end{remark}

{\bf Acknowledgements}.
This work was carried out while the authors were visiting ICTP - International Centre for Theoretical Physics, as associates. T.C. and I.N. would like to acknowledge support from the ICTP through the Associates Programme (2018-2023 and 2019-2024, respectively). The authors are very grateful to the Institute and also to Claudio Arezzo for the hospitality. The authors also would like to thank Lucas Ambrozio for his interest in this work and for several enlightening comments. T. C. has been partially suported by  CNPq/Brazil grant 311803/2019-9 and T.C. and I.N. were partially supported by the Brazilian National Council for Scientific and Technological Development  (CNPq Grant 405468/2021-0).

\section{Preliminaries}

We start this section with the following basic remark: equations \eqref{staticequations2} are equivalent to the equations
\begin{equation}\label{staticequations3}
\Hess_gV=V\left(\Ric_g - \dfrac{R_g}{n-1}g\right)\ \ \mbox{and} \ \ \Delta_gV+\dfrac{R_g}{n-1}V=0 \ \ \mbox{on $M$}
\end{equation}
and 
\begin{equation}\label{staticequations4}
V\left(\Pi_g-\dfrac{H_g}{n-1}g\right)=0 \ \ \mbox{and} \ \ \dfrac{\partial V}{\partial \nu}=\dfrac{H_g}{n-1}V \ \ \mbox{on $\partial M$}.
\end{equation}
In order to see this equivalence, take the trace of equations \eqref{staticequations2}.

Next, we present some properties of static manifolds with boundary.

 \begin{proposition}\label{properties} 
 Let   $(M^n,g)$ be a static Riemannian manifold with boundary with static potential $V: M\to \mathbb R$. Suppose that $\Sigma=V^{-1}(0)$ is nonempty. Then:
\begin{itemize}
\item[(a)] $\Sigma$ is an embedded totally geodesic hypersurface in $M$. More precisely:
\begin{itemize}
\item[(a.1)] If $\Sigma\cap \partial M = \emptyset$, then $\Sigma$ is a totally geodesic hypersurface contained in $\inte M$;
\item[(a.2)] If $\Sigma\cap \partial M\neq \emptyset $, then each connected component $\Sigma_0$ of $\Sigma$ such that $\Sigma_0\cap \partial M\neq \emptyset$ is a free boundary totally geodesic hypersurface of $M$. In particular, $\partial \Sigma_0$ is a totally geodesic hypersurface of $\partial M$.
\end{itemize}
\item[(b)]$\kappa:=|\nabla_M V|$ is a positive constant
on each connected component of $\Sigma;$
\item[(c)] The scalar curvature $R_g$  is constant;
\item[(d)] $\partial M$ is totally umbilical and its mean curvature $H_g$ is constant on each connected component of $\partial M$. In particular, it follows from item (a.2) that if $\Sigma_0$ is a connected component of $\Sigma$ such that $\Sigma_0\cap \partial M\neq\emptyset $, the mean curvature of $\partial\Sigma_0$ in $\Sigma_0$ is locally constant.
\item[(e)] On $\partial M$, we have
$$
\Ric_g(\nu,X)=0,
$$
for any vector $X$ tangent to $\partial M$, where $\nu$ denotes the unit conormal vector field to $\partial M$.
%\item[d)] Consider the following functional  on the space of the Riemannian metrics
%\begin{align*}
%\mathcal{F}(g)=& \int_{ M} R_g V\;dv+ 2\int_{\Sigma} H_g V\; da,
%\end{align*}
%where $V$ is a given smooth nontrivial function  on $M.$  iF $M$ is compact, then $g$ is a critical point of $\mathcal{F}.$         
\end{itemize}
\end{proposition}

\begin{proof}

It is well known that $0$ is a regular value of $V:M\to \mathbb{R}$ (see, for example, \cite{FM, Cor, A}). In order to prove (a), we need only to prove that $0$ is also a regular value of the restriction $V:\partial M\to \mathbb{R}$. 

In fact, we claim that if  there exists a point  $x_0\in \partial M$  such that $V(x_0)=0,$  then  $\nabla_{\partial M}V(x_0)\neq 0.$ By contradiction, suppose that $\nabla_{\partial M}V(x_0)=0$. Since $\partial V/\partial \nu=H_g/(n-1)V$ on $\partial M$ and $V(x_0)=0$, we conclude that $\partial V/\partial\nu (x_0)=0$. Thus $\nabla_M V(x_0)=0$.

Let $\gamma:[0,\varepsilon)\to M$ be a geodesic of $M$ such that $\gamma(0)=x_0$ and $\gamma(t)\in \inte M$ for $t>0$. Define, $h(t)\vcentcolon=V(\gamma(t))$. As in \cite{FM} (see also \cite{Cor}), by using \eqref{staticequations1} we have that $h$ solves 
$$
h^{\prime\prime}(t)=\left(\Ric_g(\gamma^\prime(t),\gamma^\prime(t)-\dfrac{R_g}{n-1}|\gamma^\prime(t)|^2\right)h(t).
$$
Since $h(0)=h^\prime(0)=0$, we have by uniqueness for solutions to second-order ODE's that $h\equiv 0$, that is, $V$ vanishes on $\gamma$. By varying the initial condition $\gamma^\prime(0)\in T_{x_0}M$, we can conclude that there exists an open set $U\subset \inte M$ such that $V$ vanishes on $U$. As $V$ solves $(n-1)\Delta_g V=R_g V$ on $M$, we have by analytic continuation that $V=0$ on $M$ which is a contradiction. Thus, $0$ is a regular value of $V:\partial M\to \mathbb{R}$.

%We define  $ h(t) := V(\alpha(t)),$ where $\gamma$ 
%is any geodesic of $\partial M$ starting from $x_0$. A straightforward calculation shows that $h(t)$ satisfies the following linear second-order ODE:
%\begin{align*}
%h''&(t)= (\Hess_g V{|_{\partial M}})_{\alpha(t)}\cdot(\alpha'(t),\alpha'(t))\\
%&=(\Hess_g V)_{\alpha(t)}\cdot(\alpha'(t),\alpha'(t))-\Pi_{g}(\alpha'(t),\alpha'(t))d(V\circ\alpha)(\nu)\\
%& = \Big(\Ric_g(\alpha'(t),\alpha'(t))-\frac{R_g}{n-1}\|\alpha'(t)\|^2_g-\frac{H_{g}}{n-1}\Pi_g(\alpha'(t),\alpha'(t))\Big)h(t).
%\end{align*}
%Since $h(0)=0$ and $h'(0)=0,$ we have  by uniqueness for solutions to second order ODE’s that  $h(t)\equiv 0$ and this implies that $V$ vanishes on a neighborhood $U$ of $x_0$ in $\partial M$.  
 %As before,  define another function $f(t)=V\circ\gamma(t),$ where $\gamma$ is any geodesic intersecting orthogonally $\partial M$ starting at a point $q\in U$.  Since $V=0$  on $U,$ we have that $\partial V/\partial \nu(q)=0$ by \eqref{staticequations2}, which implies that $f'(0)=0.$ Therefore we again obtain a second order ODE with initial data $f(0)=0$ and $f'(0)=0$, and this implies that $f(t)\equiv 0$. Thus, $V$ vanishes on a neighborhood of $x_0$ in $M$. Since $V$ solves $(n-1)\Delta_g V=R_gV$ on $M$, we have by analytic continuation that $V=0$ on $M$ which is a contradiction. Thus, $0$ is a regular value of $V:\partial M\to\mathbb{R}$.

Let $\Sigma_0$ be a connected component of $\Sigma$ such that $\Sigma_0\cap \partial M\neq\emptyset$. To see that $\Sigma_0$ meets $\partial M$ orthogonally along $\partial\Sigma_0$ we have to note that $\langle \nabla_MV(q),\nu(q)\rangle=0$ and that $\nabla_MV(q)$ is orthogonal to $\Sigma_0$ for all $q\in\partial \Sigma_0\cap \partial M$.

Moreover,  since $\Hess_{g} V(X,Y)=0$ and $\Hess_{g} V(X,\nabla_M V)=0$ for all tangent vectors $X, Y$  to $\Sigma,$ we conclude that $\Sigma$ is totally geodesic and $|\nabla_M V|^2$ is a constant along $\Sigma$ (also see \cite{Cor, A}). 

The proof that $R_g$ is constant is well known (see, for example, \cite{FM, Cor, A}).
It follows from \eqref{staticequations4} that $\Pi_{g}=\frac{H_{g}}{n-1}g$ on $\partial M$ since $V$ vanishes only on a set of zero measure of $\partial M$.

Consider $q\in \partial M$ and $X\in T_q(\partial M)$. At $q\in \partial M$ we have that 
\begin{eqnarray*}
\Hess_g V(X,\nu)&=& X\langle\nabla_M V,\nu\rangle-\langle\nabla_M V,\nabla_X\nu\rangle\\
&=& X\left(\frac{\partial V}{\partial \nu}\right) -\Pi(\nabla_{\partial M} V,X)\\
&=& X\left(\dfrac{H_g}{n-1}V\right) -\dfrac{H_g}{n-1}X(V)\\
&=& \dfrac{X(H_g)}{n-1} V.\\
\end{eqnarray*}

Now, by the contracted Codazzi equation
$$
\Div_{\partial M}\Pi_g-dH_g=\Ric_g(\nu,\cdot),
$$
we obtain that 
$$\Ric_g(\nu,X)=\frac{n}{n-1} X(H_g).$$

Since $V\Ric_g(X,\nu)=\Hess_g V(X,\nu)$ by \eqref{staticequations2} we have that $X(H_g)V=0$ on $\partial M$.  Therefore, $dH_g=0$ on $\partial M$ and this implies that $H_g$ is constant. In particular, on $\partial M$, we have that $\Ric_g(\nu,X)=0$ for any $X$ tangent to $\partial M$.

Finally, let $\Sigma_0$ be a connected component of $\Sigma$  such that $\Sigma_0\cap \partial M\neq \emptyset$. Since $\Sigma_0$ meets $\partial M$ orthogonally along $\partial\Sigma_0$, it is not difficult to see that the principal curvatures of $\partial\Sigma_0$ in $\Sigma_0$ at a point $q$ are all equal to $H_g(q)/(n-1)$.
\end{proof}

\begin{remark}
We note that in the case $R_g=0$, it follows from \eqref{staticequations3} and \eqref{staticequations4} that $V$ is an Steklov eigenfunction of $M$ with eigenvalue equal to $H_g/(n-1)$ (see also \cite{HH}, Proposition 3.8).
\end{remark}

Given a compact static manifold with boundary $(M^3,g)$, the following proposition gives an area estimate for the connected components of $\partial M$ where the potential $V$ is nonzero and does not change sign. 

\begin{proposition}\label{prop_areaestimate}
Let $(M^3,g)$ be a compact static manifold with boundary with static potential $V:M\to \mathbb{R}$. Let $S$ be a connected component of $\partial M$ and suppose that $V\neq 0$ everywhere on $S$. Then
\begin{equation}\label{areaineq}
\left(\dfrac{R_g}{2}+\dfrac{3}{4}H_S^2\right)|S|\leq 2\pi \chi(S),
\end{equation}
where $H_S$ denotes the mean curvature of $S$ with respect to the outward unit normal vector field $\nu$. Equality holds if and only if 
\begin{itemize}
    \item[(i)] $V$ is constant on $S$;
    \item[(ii)]  $S$ has Gauss curvature $K_S=\dfrac{R_g}{2}+\dfrac{3}{4}H_S^2$ and $\Ric_g(\nu,\nu)=-\dfrac{H_S^2}{2}$ on $S$.
\end{itemize}
\end{proposition}
\begin{proof}
On $S$ we have that 
\begin{equation}\label{eq1}
\Delta_M V=\Delta_SV+H_S\dfrac{\partial V}{\partial \nu} +\Hess_MV(\nu,\nu).
\end{equation}

By using the static equations \eqref{staticequations3} and \eqref{staticequations4} we get from \eqref{eq1} that
\begin{equation}\label{eq2}
\Delta_SV+\left(\Ric_g(\nu,\nu)+\dfrac{H_S^2}{2}\right)V=0.
\end{equation}

Therefore, since $V\neq 0$ everywhere on $S$, we obtain 
\begin{equation}\label{eq3}
\dfrac{\Delta_S V}{V}+\dfrac{R_g}{2}+\dfrac{3}{4}H_S^2-K_S=0,   
\end{equation}
where we have used the contracted Gauss equation on $S$ which is given by $2\,\Ric_g(\nu,\nu)=R_g-2K_S+H_S^2-|A_S|^2$. Here, $A_S$ denotes the  second fundamental form of $S$, respectively.

By using integration by parts, we have that
\begin{equation}\label{eq4}
\int_S\dfrac{|\nabla_SV|^2}{V^2}\,d\sigma+\int_S\dfrac{R_g}{2}d\sigma+\dfrac{3}{4}H_S^2\,d\sigma=\int_S K_S\,d\sigma,
\end{equation}

By the Gauss-Bonnet theorem, \eqref{eq4} implies  
\begin{equation}\label{eq5}
\left(\dfrac{R_g}{2}+\dfrac{3}{4}H_S^2\right)|S|\leq 2\pi \chi(S)
\end{equation}

It follows from \eqref{eq4} that equality holds in \eqref{eq5} if and only if $V$ is constant on $S$. In this case, by \eqref{eq2} and \eqref{eq3}, we conclude that 
$$
K_S=\dfrac{R_g}{2}+\dfrac{3}{4}H_S^2\ \  \mbox{and} \ \ \Ric_g(\nu,\nu)=-\dfrac{H_S^2}{2} \ \ \mbox{on $S$}.
$$
\end{proof}
\begin{corollary}\label{cor_areaestimate}
Let $(M^3,g)$ be a compact static manifold with boundary with static potential $V:M\to \mathbb{R}$. Let $S$ be a connected component of $\partial M$ and suppose that $V\neq 0$ everywhere on $S$. 
\begin{itemize}
    \item[(i)] If $R_g=0$, then either $S$ is a totally geodesic two-torus or it is a totally umbilical two-sphere. If $H_S>0$, then $S$ is in fact a totally umbilical two-sphere.
    \item[(ii)] If $R_g>0$, then $S$ is a totally umbilical two-sphere.
\end{itemize}
\end{corollary}

%In any of the above cases (i) and (ii), the first eigenvalue of the Jacobi operator $L=\Delta_S+\Ric(\nu,\nu)+H_S^2/2$ of $S$ if equal to $\lambda=0$ and $V$ is a first eigenfunction of $L$ (see Equation \eqref{eq2}).

\begin{remark}
In the case $R_g=0$ and $H_S>0$, we note that equality in \eqref{areaineq} is attained on the space-like  slice  of the photon sphere in the static spherically symmetric Schwarzchild black hole spacetime with mass $m=2/(3\sqrt{3}H_S)$ (see Example \ref{exampleSchw}).
\end{remark}
%\section{Examples}

\section{Proof of Theorem \ref{shentypetheorem}}

%In order to prove Theorem \ref{shentypetheorem} we will need the following result.%

In order to prove Theorem \ref{shentypetheorem} we will need the following Pohozaev-type integral identity due to Schoen \cite{S}.

\begin{theorem}[Schoen \cite{S}]
Let $(M^n,g)$ be a compact Riemannian manifold with boundary. If $X$ is a smooth vector field on $M$, then
\begin{equation}\label{pohoschoentypeid}
\dfrac{n-2}{2n}\int_M X(R_g)\,dv = -\dfrac{1}{2}\int_M\langle\mathcal{L}_Xg,\Ricst_g\rangle\,dv +\int_{\partial M} \accentset{\circ}{\Ric}(X,\nu)\,d\sigma,   
\end{equation}
where $R_g$ and $\Ricst$ denote the scalar curvature and the trace free part of the Ricci tensor of $M$, respectively. In addition,  $dv$ and $d\sigma$ stands for the volume elements of $M$ and $\partial M$, respectively.
\end{theorem}

We note that a general version of \eqref{pohoschoentypeid} was deduced by Gover and Orsted in \cite{GO}. 
%\begin{theorem}[Gover and Orsted \cite{GO}]
%Let $(M^n,g)$ be a compact Riemannian manifold with boundary. If $B$ is a divergence free symmetric two-tensor and $X$ is a smooth vector field on $M$, then
%\begin{equation}\label{pohozaevtypeid}
%\int_M X(b)\,dv = \dfrac{n}{2}\int_M\langle\mathcal{L}_Xg,\accentset{\circ}{B}\rangle\,dv - n\int_{\partial M} \accentset{\circ}{B}(X,\nu)\,d\sigma,   
%\end{equation}
%where $\accentset{\circ}{B}$ and $b=\tr_gB$ are the trace free part of $B$. In addition,  $dv$ and $d\sigma$ stands for the volume elements of $M$ and $\partial M$, respectively.
%\end{theorem}

\begin{proof}[Proof of Theorem \ref{shentypetheorem}]

%Let $V:M\to \mathbb{R}$ be the static potential of $M$. Since $R_g=0$, we have by  \eqref{staticequations2} that $\Delta_MV=0$ on $M$ and $\partial V/\partial \nu= V$ on $\partial M$, where $\nu$ is the outward unit conormal to $\partial M$. Thus
%$$
%0=\int_M \Delta_M V = \int_{\partial \Sigma} V.
%$$
%This implies that $V$ changes sign in $\partial M$. In particular, by Proposition \ref{properties}, item (a), $\Sigma=\{V=0\}$ is free boundary totally geodesic surface in  $M$.

Let $V:M \to \mathbb{R}$ be the static potential of $M$ and  $\Sigma=V^{-1}(0)$ its zero set. First, suppose that $\Sigma\cap \partial M=\emptyset$. Thus, by Proposition \ref{properties}, item (a), $\Sigma$ is a closed totally geodesic surface in $\inte M$. 

Let $\Omega$ a connected component of $M\setminus \Sigma$ and let $S$ denote the set $\partial \Omega\setminus \Sigma$. We may assume that $V>0$ on $\Omega$. Denote by $\nu$ and $\xi$ the outward unit normal vector field to $S$ and $\Sigma$, respectively.

Since $R_g=0$, we have that $\Ric_g$ is trace free. Moreover, note that $\mathcal{L}_Xg=2\Hess_g V$. Thus, since $\Hess_gV=V\Ric_g$, by applying \eqref{pohoschoentypeid} on $\Omega$ we get
\begin{eqnarray*}
\int_\Omega V|\Ric_g|^2\,dv&=&\int_\Sigma \Ric_g(\xi,\nabla_MV)\,d\sigma + \int_S\Ric_g(\nu,\nabla_MV)\,d\sigma \\
& = & -  \int_\Sigma|\nabla_M V|\Ric_g(\xi,\xi)d\sigma + \int_S V\Ric_g(\nu,\nu)\,d\sigma,
\end{eqnarray*}    
where we have used that $\xi=-\nabla_M V/|\nabla_MV|$ on $\Sigma$ and Proposition \ref{properties}, item (e), together with the fact that $\nabla_MV=\nabla_SV+V\nu$ on $S$.

By the contracted Gauss equation on $\Sigma$, we have $\Ric_g(\xi,\xi)=-K_\Sigma$. Moreover, by \eqref{staticequations2} we have $V\Ric_g(\nu,\nu)=\Hess_gV(\nu,\nu)$ on $S$. Thus, since $\Delta_MV=0$ on $\Omega$ we get that
$$
V\Ric_g(\nu,\nu)=-\Delta_SV-2\dfrac{\partial V}{\partial\nu} \ \ \mbox{on $S$.}
$$

Therefore,
\begin{eqnarray*}
0 & \leq & \int_\Omega V|\Ric_g|^2\,dv\\
&=& \kappa \int_\Sigma K_\Sigma\,d\sigma - \int_S \Delta_SV\,d\sigma - 2\int_S \dfrac{\partial V}{\partial \nu}\,d\sigma\\
&= & \kappa\int_\Sigma K_\Sigma\,d\sigma -2\int_\Omega \Delta_MV\,dv + 2\int_\Sigma\dfrac{\partial V}{\partial\xi}\,d\sigma\\
%&=& \kappa\int_\Sigma K_\Sigma\,d\sigma - 2\kappa|\Sigma|\\
& = & 2\kappa\left(\pi\chi(\Sigma) - |\Sigma|\right),
\end{eqnarray*}
where we have used that $\partial V/\partial \xi=-\kappa$ on $\Sigma$.

Since $\kappa>0$, we conclude that $|\Sigma|\leq \pi\chi(\Sigma)$. In particular, $\Sigma$ is a two-sphere and $|\Sigma|\leq 2\pi$. In fact, we have $|\Sigma|<2\pi$ since otherwise $\Omega$ would be Ricci flat. More precisely, if $|\Sigma|=2\pi$, then the above computation shows that 
$$
\int_\Omega V|\Ric_g|^2\,dv=0
$$
for any connected component $\Omega$ of $M\setminus \Sigma$. Therefore, $\Ric_g=0$ on $M$ which implies that $(M,g)$ is flat. Since $\Sigma$ is totally geodesic in $M$ we obtain that $\Sigma$ is flat. By the Gauss-Bonnet theorem, this is a contradiction as $\Sigma$ is a two-sphere.

Now, suppose that $\Sigma\cap \partial M\neq \emptyset$. In this case, by Proposition \ref{properties}, item (a), $\Sigma$ is a free boundary totally geodesic surface in $M$.

Let $\Omega$ be a connected component of $M\setminus \Sigma$ and let $S$ denote the set $\partial\Omega\setminus \Sigma$. Again, we may assume that $V>0$ on $\Omega$. Note that on $\Gamma=S\cap \Sigma$, we have $S\perp \Sigma$. Denote by $\nu$ and $\xi$ the outward unit normal vector field to $S$ and $\Sigma$, respectively.

Similar to the the previous case, we get that
\begin{eqnarray*}
0&\leq& \int_\Omega V|\Ric_g|^2\,dv\\
&=& \kappa \int_\Sigma K_\Sigma\,d\sigma - \int_S \Delta_SV\,d\sigma - 2\int_S \dfrac{\partial V}{\partial \nu}\,d\sigma\\
&= & \kappa\int_\Sigma K_\Sigma\,d\sigma-\int_\Gamma \dfrac{\partial V}{\partial\xi} -2\int_\Omega \Delta_MV\,dv + 2\int_\Sigma\dfrac{\partial V}{\partial\xi}\,d\sigma\\
%&=& \kappa\left(\int_\Sigma K_\Sigma\,d\sigma +|\Gamma|\right)- 2\kappa|\Sigma|\\
& = & 2\pi\kappa\chi(\Sigma)-2\kappa |\Sigma|,
\end{eqnarray*}
where we have used the Gauss-Bonnet theorem for surfaces with boundary and that the geodesic curvature of $\Gamma$ in $\Sigma$ is equal to $1$ and  $\partial V/\partial\xi=-\kappa$ on $\Sigma$. Therefore, we have that $|\Sigma|\leq \pi$.

Moreover, if we suppose that $|\Sigma|=\pi$ then we conclude, as in the previous case,  that $(M^3,g)$ is Ricci flat which implies $(M^3,g)$ is flat. In particular, $\partial M$ is connected. Since $\Pi_g=g$ we obtain by the Gauss equation that $\partial M$ has constant Gauss curvature equal to $1$. Thus, $\partial M$ with the induced metric is isometric to the round sphere $(\mathbb S^2,g_{can})$ of radius $1$ and this implies that $(M^3,g)$ is isometric to the Euclidean unit ball $(B^3,\delta)$. In particular, by \cite{HH}, Proposition 4.2, we conclude that $V$ is given by $V(x)=\langle x,v\rangle$, $x\in \Omega=B^3$, for some vector $v\in\mathbb{R}^{n}\setminus\{0\}$.
\end{proof}

\section{Proof of Theorem \ref{photonsphererigidity}}

In this section we will prove Theorem \ref{photonsphererigidity}. Let $(M^3,g)$ be an orientable compact static manifold with boundary with static potential $V:M\to \mathbb{R}$. Suppose that $R_g=0$ on $M$, $H_g>0$ on $\partial M$, $\Sigma=V^{-1}(0)$ is connected and $\Sigma=V^{-1}(0)\subset \inte M$. %We are not assuming here that $\Sigma$ is connected.

Let $\Omega$ be a connected component of $M\setminus \Sigma$. We may assume that $V>0$ on $\Omega$. We have $\partial \Omega=\Sigma\cup \left (\cup_{i=0}^kS_i\right)$, where each $S_i$ is a connected surface with positive mean curvature. In fact, by Corollary \ref{cor_areaestimate}, item (i), each $S_i$ is a totally umbilical two-sphere.

First, we will prove that $S=\cup_{i=0}^k S_i$ is connected, that is, $k=0$. Following Section 6 of \cite{A}, let $U$ be a small tubular neighborhood of $\Sigma$ diffeomorphic to $[0,1)\times \Sigma$. We identify any point $p\in U$ with its corresponding point $(r,x)\in [0,1)\times\Sigma$, that is, we write $p=(r,x)$. Let $N^4$ be the quotient space given by 
\begin{equation}
N^4=\mathbb S^1\times (\Omega\setminus\Sigma)\,\sqcup\,(B^2_1\times \Sigma)/\sim,
\end{equation}
where $\mathbb S^1$ and $B^2_1$  are the unit circle and the open unit ball centered at the origin of $\mathbb{R}^2$, respectively, and $\sim$ is the equivalence relation that identifies $(\theta,p)\in \mathbb S^1\times (U\setminus\Sigma)$ with $(r\cos\theta, r\sen\theta, x)\in B^2_1\times \Sigma$ if $p=(r,x)$. 

The smooth structure on $\mathbb S^1\times (\Omega\setminus\Sigma) \sqcup (B^2_1\times \Sigma)$ induces a canonical smooth structure on $N^4$ and with respect to this smooth structure $\{0\}\times \Sigma$ is a smooth embedded submanifold of codimension two that we can identify with $\Sigma$. In our case, differently from \cite{A}, $N^4$ has nonempty boundary. In fact, $\partial N=\mathbb S^1\times S$.

The dense open subset $N^4\setminus \Sigma$ can be identified with $\mathbb S^1\times (\Omega\setminus\Sigma)$. As $V>0$ on $\Omega\setminus \Sigma$, we can define the smooth Riemannian metric $h=V^2d\theta^2+g$ on $N^4\setminus \Sigma$.
%\begin{equation}
%h=V^2d\theta^2+g.
%\end{equation}

The Riemannian metric $h$ is singular on $\Sigma$ and has, up to pertubation, a conical behaviour with angle $k=|\nabla_{\Sigma}V|>0$ in the directions transverse to $\Sigma$. Moreover, since $(\Omega,g)$ is a static manifold with potential $V$ we have that $h$ is Einstein. The manifold $N^4$ endowed with the singular metric $h$ is called the singular Einstein metric associated to the static manifold $(\Omega,g)$ with static potential $V:\Omega\to \mathbb{R}$ (see Section 6 of \cite{A} for details). In our case, $h$ is Ricci flat since $(\Omega,g)$ is scalar flat.  Since $\Sigma$ is compact and connected,  we  normalize the static potential $V$ in such way that $|\nabla V| =1$ on $\Sigma$. In this case, the metric $h$ extends smoothly to $\Sigma$ (see for instance  \cite{Se} and \cite{A}, Section 6).

Now, let $\partial_iN=\mathbb S^1\times S_i$ be a connected component of $\partial N$. A simple computation shows that $\partial_iN$ is totally umbilical with principal curvatures equal to 
$$
\frac{1}{V}\frac{\partial V}{\partial \nu}=\dfrac{1}{2}H_{S_i},
$$
where $\nu$ denotes the outward unit normal vector to $S_i$ and $H_{S_i}=H_g|_{S_i}$. Since $H_g>0$ on $S$, we have that $\partial_iN$ has positive mean curvature for all $i=0,1,\ldots,k$.

Since $h$ is Ricci flat and smooth on $N$ and every connected component of $\partial N$ has positive mean curvature it follows from \cite{FL}, Proposition 2.8, that $\partial N$ is connected. %In fact, %since $\overline{\Omega}=\Omega\cup \Sigma$ is compact and $\partial N$ is strictly convex with respect to the outward unit normal vector we have that $N^4$ with {\color{red}the distance function induced by $h$ is a complete length space and the Hopf-Rinow theorem holds}.   
%suppose, by contradiction, that $\partial N$ is not connected and let $\partial_iN$ and $\partial _jN$ be two distinct connected components of $\partial N$. Let $\gamma:[0,l]\to N$, with $\gamma(0)\in\partial_iN$ and $\gamma(l)\in\partial_jN$, be a minimizing geodesic realizing the distance between $\partial _iN$ and $\partial_jN$. Since $\partial N$ is strictly convex, the interior of $\gamma$ does not touch $\partial N$. Moreover, $\gamma$ meets $\partial_iN$ and $\partial_jN$ orthogonally. Since $h$ is Ricci flat and $\partial_iN$ and $\partial_jN$ have positive mean curvature we get a contradiction from a standard argument by using the second variation formula of arc length of $\gamma$. Finally, the connectedness of $\partial N$ implies that $S$ is connected. This proves item (i) of Theorem \ref{photonsphererigidity}.

In order to prove item (ii), first note that Proposition \ref{psareaestimate} implies that
\begin{equation}\label{areaestimate3}
H_S^2|S|\leq \dfrac{16\pi}{3}.
\end{equation}

Suppose that equality holds in \eqref{areaestimate3}. Again, by Proposition \ref{psareaestimate} we have that $V$ is constant on $S$ and that $K_S=(3/4)H_S^2.$

Define $m=2/(3\sqrt{3}H_S)>0$. Let $(M_{3m}=[3m,+\infty)\times \mathbb{S}^2,g_m)$, where 
$$
g_m=\left(1-\dfrac{2m}{r}\right)^{-1}dr^2+r^2\overline{g}
$$
and $\overline{g}$ is the standard metric of the two-sphere $\mathbb{S}^2$ of radius 1 (see Example \ref{exampleSchw}). $(M_{3m},g_m)$ is the complement of the region on the three-dimensional Schwarschild space bounded by the photon sphere $S_{3m}=\{3m\}\times \mathbb{S}^2=\partial M_{3m}$. Note that $V_m(r)=(1-2m/r)^{1/2}$ defines a static potential on $(M_{3m},g_m)$. 

By our choice of $m$, we have that $K_{S}=K_{S_{3m}}$. Therefore, we conclude that $(S,g|_{S})$ and $(S_{3m},g_{m}|_{S_{3m}})$ are isometric.  Thus, by gluing $(\Omega,g)$ and $(M_{3m},g_m)$ along $S$ and $S_{3m}$ we can find a Riemannian manifold $(\widetilde{M}^3_+,\widetilde{g})$ where the metric is given by $\widetilde{g}=g$ on $\Omega$ and $\widetilde{g}=g_m$ on $M_{3m}$. Note that $\widetilde{g}$ is smooth away from the gluing surface.

After a rescaling of $V$, if necessary, we can assume that $V|_{S}=V_m|_{S_{3m}}$. As done in \cite{CG1}, Section 3, Step 1, we can conclude that the metric $\widetilde{g}$ is $C^{1,1}$ across the gluing surface.

Next, define $\widetilde{V}:\widetilde{M}_+^3\to\mathbb{R}$ by 
$$
\widetilde{V}(p)=\left\{
\begin{array}{rll}
V(p), & \mbox{if } & p\in \Omega    \\
V_m(p),     & \mbox{if} & p\in M_{3m}. 
\end{array}
\right.
$$
Since $V$ and $V_{3m}$ satisfy the boundary conditions \eqref{staticequations4} on $S$ and $S_{3m}$, respectively, and $V|_{S}$ and $V_m|_{S_{3m}}$ we have that $\widetilde{V}$ is smooth away from the gluing surface and $C^{1,1}$ across it.  Now, reflect $(\widetilde{M}^3_+,\widetilde{g})$ through $\Sigma$ to obtain an isometric copy $(\widetilde{M}^3_-,\widetilde{g})$ and glue them to each other along $\Sigma$ to obtain a new manifold $(\widetilde{M}^3,\widetilde{g})$ which has two isometric ends and is smooth away $S$ (and its reflected copy) and $\Sigma$ and $C^{1,1}$ across them.

Define $\widetilde{V}_+:=\widetilde{V}$ and $\widetilde{V}_-:=-\widetilde{V}.$ By a  abuse of notation we again denote $\widetilde{V}=\widetilde{V}_\pm$ on $\widetilde{M}^3_\pm$.
 
As in \cite{BuMa},  we wish to use  $u:=(\widetilde{V}+1)/2$ on $\widetilde{M}$ as conformal factor. 
Before we have to show that $u>0$ on $\widetilde{M}$. It suffices to show that  $0\leq \widetilde V<1$ on $\widetilde M^3_+.$ In order to prove this, note that $V$ is harmonic on $\Omega$. Thus, $V$ attains its maximum on $\partial\Omega$. Since $V=0$ on $\Sigma$ and $V>0$ on $\Omega\setminus\Sigma$ we have that $V$ attains its maximum at $S$. Therefore, $V\leq \max_{S}V=\max_{S_{3m}} V_m<1$. Finally, since $V_{3m}<1$ on $M_{3m}$ we conclude that $0\leq \widetilde{V}<1$ on $\widetilde{M}^3_+.$%só usar o Maximum principle.

We define  the Riemannian metric 
$
\widehat{g}=u^4\widetilde g
$
in such way that   $(\widetilde{M}^3_+,\widehat{g})$ has  ADM-mass equal to zero and the doubled end can be one-point compactified, that is,  $(\widetilde{M}^3_-,\widehat{g})$ can be compactified by adding in a point $\infty$ at infinity. By construction, we obtain  a  geodesically complete, scalar flat Riemannian manifold and one asymptotically flat end of zero ADM mass that is smooth away from $\Sigma$ and $S$ and one point.

By the rigidity statement of the  (weak regularity) Positive Mass Theorem  proved by Bartnik \cite{Ba}, the conformally modified manifold $(\widehat{M}^3,\widehat{g})$ must be isometric to the Euclidean space. Therefore we have that $(\widetilde M^3,\widetilde g)$  must be conformally flat and, as in \cite{BuMa}, we conclude that  $(\widetilde M^3,\widetilde g)$  is the Schwarzschild solution. Hence $(\Omega,g)$  is isometric to $([2m,3m]\times \mathbb S^2,g_m)$.

\end{document}